\documentclass[a4paper,12pt]{article}
\usepackage{soul}
\usepackage{amsmath, amssymb, amsfonts, amsthm}
\usepackage{mathrsfs}
\usepackage{enumitem}
\usepackage{stmaryrd}
\usepackage[latin1]{inputenc}
\usepackage[T1]{fontenc}
\usepackage{color}
\usepackage{graphicx}
\usepackage{bm}
\usepackage{hyperref}
\usepackage{wrapfig}
\usepackage{cleveref}
\usepackage{ulem}

\theoremstyle{plain}
\newtheorem{theorem}{Theorem}[section]

\newtheorem{proposition}[theorem]{Proposition}
\newtheorem{lemma}[theorem]{Lemma}
\newtheorem{corollary}[theorem]{Corollary}
\theoremstyle{remark}
\newtheorem{remark}[theorem]{Remark}

\numberwithin{equation}{section}  

\evensidemargin\oddsidemargin

\usepackage{mathrsfs, stmaryrd, mathtools}

\DeclarePairedDelimiterX\intff[2]{[}{]}{#1,#2}
\DeclarePairedDelimiterX\intfo[2]{[}{)}{#1,#2}
\DeclarePairedDelimiterX\intof[2]{(}{]}{#1,#2}
\DeclarePairedDelimiterX\intoo[2]{(}{)}{#1,#2}
\DeclarePairedDelimiter{\pars}{(}{)}
\DeclarePairedDelimiter{\bracks}{[}{]}

\DeclarePairedDelimiter{\braces}{\lbrace}{\rbrace}
\DeclarePairedDelimiterX{\setof}[2]{\lbrace}{\rbrace}{#1\,{:}\,#2}
\DeclarePairedDelimiterX{\bracksof}[2]{[}{]}{#1\,\delimsize\vert\,#2}
\DeclarePairedDelimiterX{\parsof}[2]{(}{)}{#1\,\delimsize\vert\,#2}
\DeclarePairedDelimiterXPP\lnorm[2]{}\lVert\rVert{_{#1}}{#2}

\newcommand{\green}[1]{G_{#1}}
\newcommand{\rwS}[1]{S_{#1}}
\newcommand{\rwD}{\eta}
\newcommand{\rwP}[2]{\mathsf P^{#1}_{#2}}
\newcommand{\rwE}[2]{\mathsf E^{#1}_{#2}}
\newcommand{\rwVar}[1]{\Gamma_{#1}}
\newcommand{\rwJ}[1]{J_{#1}}
\newcommand{\capa}{\mathtt{cap}_{\rwD}}

\newcommand{\brwCh}{\mu}
\newcommand{\brwD}{\theta}

\newcommand{\tree}{{\mathcal T}}
\newcommand{\gnNd}{u}
\newcommand{\dist}{{\tt d}}

\newcommand{\gnR}{R}

\newcommand{\brwNd}[1]{V_{{#1}}}
\newcommand{\hdX}[1]{V_{u_{#1}}}

\newcommand{\qP}[1]{\mathbb P_{#1}}
\newcommand{\qE}[1]{\mathbb E_{#1}}
\newcommand{\hdP}{\qP{\brwCh,\brwD}}
\newcommand{\hdE}{\qE{\brwCh,\brwD}}
\newcommand{\spine}{w}
\newcommand{\seq}[1]{u_{#1}}

\newcommand{\qzeta}{\zeta}
\newcommand{\qH}{H}

\begin{document}

\vglue30pt
\centerline{\large\bf Capacity of the range of branching random walks in low dimensions}

\bigskip
\bigskip

 \centerline{by}

\medskip

\centerline{Tianyi Bai  and Yueyun Hu\footnote{\scriptsize LAGA, Universit\'e Sorbonne Paris Nord, 99 avenue J-B Cl\'ement, F-93430 Villetaneuse, France.  
 Email: bai@math.univ-paris13.fr and  yueyun@math.univ-paris13.fr}}

\medskip

  \centerline{\it  Universit\'e Sorbonne Paris Nord}
 
\medskip
\centerline{\it dedicated to the 75th anniversary of Professor  Andrei M. Zubkov}
\centerline{\it  and the 70th anniversary of Professor  Vladimir A. Vatutin}

 \bigskip
\bigskip
 
{\leftskip=2truecm \rightskip=2truecm \baselineskip=15pt \small

\noindent{\slshape\bfseries Summary.}   Consider a branching random walk  $(V_u)_{u\in \tree^{IGW}}$ in $\mathbb Z^d$  with the genealogy tree $\tree^{IGW}$ formed by a sequence of i.i.d. critical Galton-Watson trees.  Let $\gnR_n $ be the set of points in $\mathbb Z^d$ visited by $(V_u)$ when the index $u$ explores  the first $n$ subtrees in $\tree^{IGW}$. Our main result states that for $d\in \{3, 4, 5\}$, the capacity of $\gnR_n$ is almost surely equal to $n^{\frac{d-2}{2}+o(1)}$ as $n \to \infty$.   
  \medskip

     \medskip
 

} 

\bigskip



\section{Introduction}

In this paper, we continue the study in \cite{bai2020capacity} on the capacity of the range of a branching random walk in $\mathbb Z^d$. 

Let $d\ge 3$ and $\rwD$ be a probability distribution on $\mathbb Z^d$. The  $\rwD$-{capacity} of a finite set $A\subset\mathbb Z^d$ (with respect to $\rwD$) is defined as
\[
\capa A:=\sum_{x\in A}\rwP{\rwD}{x}(\tau^+_A=\infty),
\]
where $\rwP{\rwD}{x}$ denotes the law of a (discrete) random walk $(\rwS{n})$ with jump distribution $\rwD$ started at $x$, and $\tau^+_A:=\inf\{n\geq 1: \rwS{n}\in A\}$ is {$(S_n)$'s} first returning time to $A$.

Let $\brwCh$ be a probability distribution on $\mathbb N$. A $\brwCh$-Galton-Watson tree starts with one initial ancestor 
which produces a random number of children according to $\brwCh$, and these children form the first generation. Then particles in the first generation produce their children independently in the same way, forming the second generation. The system goes on until infinity, or until when there is no particle in a generation. In this paper, we are interested in the critical case, i.e. the case when $\sum_{k=0}^\infty k \brwCh(k)=1$. In this case, it is well-known that the Galton-Watson tree extincts (stops with no particle in finitely many generations) almost surely.  
To avoid extinction, we consider the Galton-Watson forest defined as follows. Let $(\tree_n)_{n\ge 0}$ be a sequence of independent $\mu$-Galton-Watson trees.  As showed in \Cref{fig0}, we start with a fixed infinite ray $(\spine_n)_{n\ge 0}$ called spine, and attach $\tree_n$ to  each $\spine_n$. For  every $n\ge 1$, $\spine_{n-1}$ is considered as the parent of $\spine_n$ and the whole forest is rooted at $\spine_0$ which we denote by $\varnothing$.  As  all $\tree_n$ are finite, this Galton-Watson forest is in fact an infinite rooted  tree, denoted  by $\tree^{IGW}$. Let $\qP{\brwCh}$ be the law of $\tree^{IGW}$.

Let  $\brwD$ be a probability distribution on $\mathbb Z^d$. Given a (finite or infinite) tree $\tree$, we can define a tree-indexed random walk $(\brwNd{\gnNd})_{\gnNd\in\tree}$ in $\mathbb Z^d$ as follows:  To all edges of $\tree$ we attach i.i.d. random variables which are  distributed as $\brwD$, independent of $\tree$.  Define $V_\varnothing:=0$.  For each $u \in \tree\backslash\{\varnothing\}$, let $V_u$ be the sum of those random variables which are attached to the edges in the (unique) simple path relating $u$ to the root $\varnothing$.  Clearly $\tree$ describes the genealogy of $(\brwNd{\gnNd})_{\gnNd\in\tree}$.   We may also call $(\brwNd{\gnNd})_{\gnNd\in\tree}$ a branching random walk when its genealogy  tree is a Galton-Watson tree (or forest). 

 Denote by  $\qP{\brwCh,\brwD}$ the law  of the branching  random walk $(\brwNd{\gnNd})_{\gnNd\in\tree^{IGW}}$ when  $\tree^{IGW}$ is the Galton-Watson forest distributed as $\qP{\brwCh}$.

\begin{figure}[ht]
\centering
\includegraphics[scale=0.8]{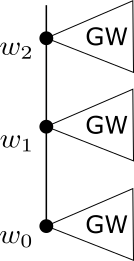}
\caption{The Galton Watson forest $\tree^{IGW}$.}\label{fig0}
\end{figure}

Under the measure $\qP{\brwCh,\brwD}$, let $\gnR_n:= \{V_u, u \in \cup_{j=0}^{n-1} \tree_j\}$ be the set of points in $\mathbb Z^d$ visited by the branching random walk $(V_u)$ when the index $u$ explores  the first $n$ subtrees of $\tree^{IGW}$. Our main result is:
\begin{theorem}\label{thm:mainlow0}
In dimensions $d=3,4,5$, let $\brwCh$ be a probability measure in $\mathbb N$, let $\brwD,\rwD$ be probability measures in $\mathbb Z^d$, with the conditions 
\begin{equation}\label{assumption}
\begin{aligned}
\left.
\begin{array}{lll}
{\mu\text{ has}}\text{ mean }1\text{ and finite variance, }{\text{and } \brwCh\not\equiv\delta_1},\\
\rwD \text{ is aperiodic, irreducible, with mean 0 and finite }(d+1) \text{-th moment},\\
\brwD \text{ is symmetric, irreducible, with some finite exponential moments.}
\end{array}
\right\}
\end{aligned}
\end{equation}
Then almost surely under $\qP{\brwCh,\brwD}$, as $n\rightarrow\infty$,
\[
\capa\gnR_n=n^{\frac{d-2}{2}+o_{\tt as}(1)},
\] where here and in the sequel,  $o_{\tt as}(1)$ denotes a quantity which converges to $0$ almost surely as $n \to \infty$. 
\end{theorem}

\begin{remark} We need the finite second moment of $\brwCh$ in \Cref{lem:max_dist} and \Cref{lm:count_pair},  and use  the symmetry and finite exponential moments  of $\brwD$  in \Cref{cor:max_v}, \Cref{lm:realG} and \Cref{lm:green_upper}, whereas the finite $(d+1)$-th moment of $\rwD$ is needed in \Cref{Green_asymptotic}. $\hfill\Box$
\end{remark}

A few comments are in order. First, it will be clear from our proof  that Theorem \ref{thm:mainlow0} holds when $\tree^{IGW}$ is replaced by a more general tree with one unique infinite ray, for example if we attach to each spine $\spine_i, i\ge 0$, an i.i.d. random number of independent $\mu$-Galton-Watson trees, as long as this random number has finite second moment. In particular Theorem \ref{thm:mainlow0} holds for the Kesten tree which is the $\mu$-Galton-Watson tree conditioned to survive forever if $\mu$ has finite third moment (because by the spine decomposition,  the number of children of $\spine_i$ in the Kesten tree has the size-biased law of $\mu$). 

Second, to avoid the extinction of a critical $\mu$-Galton-Watson tree $\tree$, we may condition $\tree$ to have $n$ vertices, thus we obtain a random tree, say $\tree_n^{cond}$. Let $R_n^{cond}:= \{V_u, u \in \tree_n^{cond}\}$ be the range of  $(\brwNd{\gnNd})_{\tree_n^{cond}}$  when the underlying genealogy tree is $\tree_n^{cond}$.  Le Gall and Lin \cite{LeGall-Lin-lowdim, LeGall-Lin-2016} studied in detail $\#R_n^{cond}$, the cardinality of the range $R_n^{cond}$,  and obtained various scaling limits for all dimensions. In particular, their results show that the critical dimension for the range of the tree-indexed walk is $d=4$:   for $d\ge 5$, $\#R_n^{cond}$ grows linearly whereas for $d=4$, $\#R_n^{cond}$ is sub-linear and for $d\le 3$, $\#R_n^{cond}$ is of order $n^{d/4}$.

The study of the capacity of the range $R_n^{cond}$ was initiated in \cite{bai2020capacity} where the authors  proved that $\capa R_n^{cond}$ grows linearly for $d\ge 7$ and is sub-linear for $d=6$. This suggests, also as conjectured in \cite{bai2020capacity},  that $d=6$ should be the critical dimension for the capacity of the range.    The main motivation of the present work is  to confirm this prediction, by giving the growth order of $\capa R_n^{cond}$ for $d\in \{3, 4, 5\}$, this will be  stated in the forthcoming Remark \ref{r:mainlow}, see \eqref{Rncond}. 

 At last, let us  mention the   systematical studies on the capacity of the range for  a simple random walk on $\mathbb Z^d$,  see Asselah, Schapira and Sousi \cite{ASS18} and the references therein. 
 
The rest of the paper is organized as follows: In Section 2, we order the vertices in the Galton-Watson forest $\tree^{IGW}$ and state the corresponding result for the range of the walk indexed by the   first $n$ vertices (Proposition \ref{thm:mainlow}).  Then \Cref{thm:mainlow0}   follows as a consequence of Proposition \ref{thm:mainlow} and  Lemma \ref{lm:population}. Sections 3 and 4 are devoted to the proofs of the upper and lower  bound of Proposition \ref{thm:mainlow} respectively.

{\bf Notation:}   Under  $\rwP{\brwD}{x}$ (resp: $\rwP{\rwD}{x}$), $(S_n)_{n\ge 0}$ denotes a random walk on $\mathbb Z^d$ starting from $x$ with  jump distribution   $\brwD$ (resp: $\rwD$). For brevity, we call $(S_n)$ a $\brwD$ (resp: $\rwD$)-random walk. Finally,  $C_i, 1\le i\le 12$ denote some positive constants.

\section{On the Galton-Watson forest}
It will be more convenient to study the capacity   for $n$ vertices  than $n$ subtrees, then we order the vertices in the Galton-Watson forest. On  $\tree^{IGW}$, we visit the vertices in the order illustrated in \Cref{fig1}:   starting with the first subtree $\tree_0$ rooted at $\spine_0$, one visits every vertex in the order of Depth-First Search (lexicographical order). Then we continue with the subtree $\tree_1$ rooted at $\spine_1$ and iterate the process. We denote the sequence of vertices in this order by $(\seq{i})_{i\ge 0}$.

\begin{figure}[ht]
\centering
\includegraphics[scale=0.5]{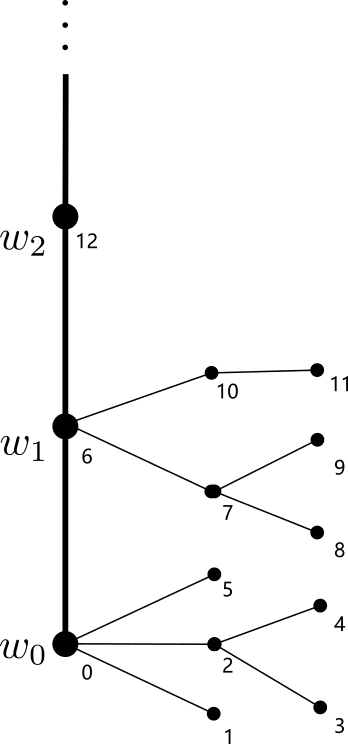}
\caption{\leftskip=1.8truecm \rightskip=1.8truecm A sample of the $\brwCh$-Galton Watson forest. The path in bold is the spine $(\spine_n)$. Labels correspond to the sequence $(\seq{i})$. For example, $u_0=\spine_0=\varnothing$ and $u_6=\spine_1$. }\label{fig1}
\end{figure}

Under the measure $\qP{\brwCh,\brwD}$, the sequence $(\seq{i})$ then induces a sequence of points in $\mathbb Z^d$, $(\brwNd{\seq{i}})$ the positions of $(\seq{i})$, and  we define
\[
\gnR[0,n]=\braces{\brwNd{\seq{0}},\brwNd{\seq{1}},\cdots,\brwNd{\seq{n}}}.
\]
The main part of this paper will be devoted to prove that
\begin{proposition}\label{thm:mainlow}
In dimensions $d=3,4,5$, let $\brwCh,\brwD,\rwD$ be probability distributions with the conditions \eqref{assumption}
Then almost surely under $\qP{\brwCh,\brwD}$,  
\[
\capa\gnR[0,n]=n^{\frac{d-2}{4}+o_{\tt as}(1)}.
\]
\end{proposition}

\begin{remark}\label{r:mainlow} As for Theorem \ref{thm:mainlow0}, Proposition \ref{thm:mainlow} also holds for more general trees with one unique infinite ray:   if we attach to each $\spine_i$ an i.i.d. random number $\nu_i$ of $\mu$-Galton-Watson tree, then the same conclusion holds as long as $\qE{\brwCh,\brwD}[\nu_i^2] < \infty$.  The proof follows in the same way as that of Proposition \ref{thm:mainlow} and we skip the details. 

  Now let $R_n^{cond}:= \{V_u, u \in \tree_n^{cond}\}$ be as before  the range of  $(\brwNd{\gnNd})_{\tree_n^{cond}}$,   where  $\tree_n^{cond}$ is the $\mu$-Galton-Watson tree conditioned to have $n$ vertices. Assume  \eqref{assumption} and furthermore that $\mu$ has finite third moment, then in probability \begin{equation}\label{Rncond}
\capa R_n^{cond}=n^{\frac{d-2}{4}+o_{\tt p}(1)},
\end{equation} where $o_{\tt p}(1)$ denotes  a quantity which converges to $0$ in probability as $n \to \infty$.  The conclusion \eqref{Rncond} follows from the aforementioned generalized version of Proposition \ref{thm:mainlow} with $\qP{\brwCh,\brwD}(\nu_i=k)= \sum_{j=k+1}^\infty\mu(j), k\ge 0,$ and   the arguments in Zhu \cite{zhu21}, Section 5 for the coupling between the  infinite tree model and $\tree_n^{cond}$.  Indeed, we first observe that $\nu_i$ has finite second moment thanks to the assumption on $\mu$. Fix $0< a <1$. By the generalized version of \Cref{thm:mainlow} (with $\nu_i$) and \cite[Lemma 3.6]{bai2020capacity}, we have  
\begin{equation*}
\capa R_n^{cond}[0,\lfloor an\rfloor]=n^{\frac{d-2}{4}+o_{\tt p}(1)},
\end{equation*}
where $R_n^{cond}[0,\lfloor an\rfloor]$ is the  range of  $(\brwNd{\gnNd})_{\tree_n^{cond}}$ when $u$ runs over the first $1+\lfloor an\rfloor$ vertices  of $\tree_n^{cond}$ in the lexicographical order. This gives a lower bound of \eqref{Rncond}, 
\[\capa R_n^{cond}\ge \capa R_n^{cond}[0,\lfloor an\rfloor]=n^{\frac{d-2}{4}+o_{\tt p}(1)}.\] 
Moreover, by exploring the tree $\tree_n^{cond}$ in the reversed order, we get  that  
\begin{equation*}
\capa R_n^{cond}[\lfloor an\rfloor,n]=n^{\frac{d-2}{4}+o_{\tt p}(1)},
\end{equation*}
yielding the upper bound because 
$\capa R_n^{cond}\le \capa R_n^{cond}[0,\lfloor an\rfloor]+\capa R_n^{cond}[\lfloor an\rfloor,n]=n^{\frac{d-2}{4}+o_{\tt p}(1)}.$ $\hfill\Box$
\end{remark}

Admitting Proposition \ref{thm:mainlow},  we deduce \Cref{thm:mainlow0}  from  the following lemma.
\begin{lemma}\label{lm:population}
Let $\brwCh\not\equiv\delta_{1}$ be a probability measure on $\mathbb N$ with mean $1$ and finite variance, then $\qP{\brwCh}$-almost surely, there are $n^{2+o_{\tt as}(1)}$ vertices in the first $n$ subtrees rooted at $\spine_0,\cdots,\spine_{n-1}$.
\end{lemma}
\begin{proof} Denote by $\#\tree$ the total vertices of a finite tree $\tree$. It is well-known (see \cite{heightbook}, Section 0.2) that there exists a random walk $Y$ on ${\mathbb Z}$ with $Y_0=0$ and jump distribution $\qP{\brwCh}(Y_1=k)= \mu(k+1)$ for $k=-1, 0, 1, 2, ...$, such that \begin{equation}\label{rwY} \# \tree_0+...+ \# \tree_{n-1}= \inf\{k\ge 1: Y_k=-n\}, \qquad n\ge 1. \end{equation}

Observe that $\qE{\brwCh}(Y_1)=0$ and $\mbox{Var}(Y_1)\in (0, \infty)$. By the classical Khintchine and Hirsch laws of iterated logarithm for the random walk $(Y_k)$ (see Cs\'aki \cite{Csaki78} for Hirsch's law of iterated logarithm under the second moment assumption), $$ - \min_{0\le k\le n} Y_k = n^{\frac12+ o_{\tt as}(1)}, \qquad \mbox{a.s.}$$

\noindent It follows that $\# \tree_0+...+ \# \tree_{n-1}= n^{2+o_{\tt as}(1)}$ a.s. 
\end{proof}

The rest of the paper is devoted to the proof of \Cref{thm:mainlow}. At first,  we need the following estimates on the population of the Galton-Watson forest $\tree^{IGW}$. For any $u, v \in \tree^{IGW}$, let $\dist(u,v)$ be the graph distance between $u$ and $v$. 

\begin{lemma}\label{lem:max_dist}
Let $\brwCh\not\equiv\delta_{1}$ be a probability measure on $\mathbb N$ with mean $1$ and finite variance, then $\qP{\brwCh}$-almost surely,
\[
\max_{0\le i\le n}\dist(\varnothing,\seq{i})= n^{\frac{1}{2}+o_{\tt as}(1)}.
\]
\end{lemma}
\begin{proof}
Let 
\begin{equation}
\qzeta_n=\max\setof{k\ge 0}{\spine_{k}\in\gnR[0,n]},\qquad\qH_n=\dist(\seq{n},\spine_{\qzeta_n}), \label{zeta}
\end{equation}
\begin{figure}[ht]
\centering
\includegraphics[scale=0.5]{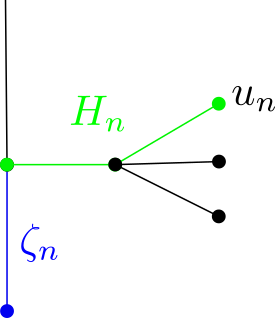}
\caption{\leftskip=1.8truecm \rightskip=1.8truecm The decomposition $\dist(\varnothing,\seq{n})=\qzeta_n+\qH_n$.}\label{figHzeta}
\end{figure}
then as showed in \Cref{figHzeta}, we have
\[
\dist(\varnothing,\seq{n})=\qzeta_n+\qH_n,\qquad \forall n\ge 0.
\]

By \Cref{lm:population}, we have 
\[\qzeta_n= n^{\frac{1}{2}+o_{\tt as}(1)}.\]

\noindent 
It thus suffices to show that $\qP{\brwCh}$-almost surely, \begin{equation}\label{eq:upper_H}
\max_{0\le i\le n}\qH_i\le n^{\frac{1}{2}+o_{\tt as}(1)}.
\end{equation}


Note that the process $(\qH_n)$ is distributed as the height process in the sense of \cite[Section 0.2]{heightbook}: Using the random walk $(Y_k)$ introduced in the proof of Lemma \ref{lm:population}, we have $$ H_n= \sum_{k=0}^{n-1} 1_{\{Y_k= \min_{k\le j \le n} Y_j\}}, \qquad n\ge1.$$

For any fixed $n$, by considering $Y_n-Y_{n-k}, 0\le k \le n$, we see that $H_n$ is distributed as $\sum_{k=1}^n 1_{\{Y_k= \max_{0\le j\le k} Y_k\}}.$ In other words,  let ${\tt t_0}:=0$ and for $j\ge 1$, ${\tt t_j}:=\inf\{k> {\tt t}_{j-1}: Y_k \ge Y_{{\tt t}_{j-1}}\}$ be the sequence of (weak) ascending  ladder epochs of $Y$. Then  for all $n, \ell \ge 1$, $$ \qP{\brwCh}(H_n\ge \ell)
= \qP{\brwCh}({\tt t}_\ell \le n)
\le
\inf_{\lambda >0} e^{\lambda n} \big(\qE{\brwCh}(e^{-\lambda {\tt t}_1})\big)^\ell,$$

\noindent where in the above inequality we have used the fact that ${\tt t}_k- {\tt t}_{k-1}, k\ge 1$ are i.i.d and distributed as ${\tt t}_1$.  The Laplace transform of $\qE{\brwCh}(e^{-\lambda {\tt t}_1})$ can be computed by the Sparre-Anderson identity,  whose asymptotic is given by  Kersting and Vatutin (\cite{Kersting-Vatutin},   proof of     Theorem 4.6, Page 75):  $$  1-  \qE{\brwCh}(e^{-\lambda {\tt t}_1}) \sim  C_1 \sqrt{\lambda}, \qquad \lambda \to 0. $$

 Take $\lambda= \frac1{n}$ we see that for all $n\ge 1$,  $\qP{\brwCh}(H_n\ge n^{\frac12}(\log n)^2) \le   e^{1 - C_2 (\log n)^2}$. It follows that   $\qP{\brwCh}(\max_{1\le k \le n} H_k \ge n^{\frac12}(\log n)^2) \le  n \,  e^{1 -C_2 (\log n)^2}$ whose sum over $n$ converges. We get \eqref{eq:upper_H}  by the Borel-Cantelli lemma. \end{proof}

\begin{corollary}\label{cor:max_v}
Let $\brwCh,\brwD$ be probability measures satisfying \eqref{assumption}, then $\qP{\brwCh,\brwD}$-almost surely,
\[
\max_{0\le i\le n}|\brwNd{\seq{i}}|= n^{\frac{1}{4}+o_{\tt as}(1)}.
\]
\end{corollary}

\begin{proof} 

Conditionally on $u\in \tree^{IGW}$ with $\dist(\varnothing, u)=k$, $V_u$ is distributed as $S_k$, where $(S_n)_{n\ge0}$ is a $\brwD$-random walk started at $0$, i.e. a random walk in ${\mathbb Z}^d$   whose law is  $\rwP{\brwD}{0}$.   By assumption \eqref{assumption},   $\rwE{\brwD}{0}(S_1)=0$ and $S_1$ has some finite exponential moments. 

Notice that $(V_{\spine_j}, 0\le j \le \zeta_n)$ is a $\theta$-random walk on $\mathbb Z^d$, and $\zeta_n=n^{\frac{1}{2}+o_{\tt as}(1)}$, we have the lower bound
$
\max_{0\le i\le n}|\brwNd{\seq{i}}|\ge \max_{0\le j\le \zeta_n}|\brwNd{\spine_{j}}|=n^{\frac{1}{4}+o_{\tt as}(1)}
$
by the same argument as in the proof of \Cref{lm:population}.

Below we show the upper bound
$
\max_{0\le i\le n}|\brwNd{\seq{i}}|\le n^{\frac{1}{4}+o_{\tt as}(1)}.
$
Indeed, applying Petrov (\cite{petrov1995limit}, Theorem 2.7 and Lemma 2.2)  gives that for all $n\ge 1$ and $\lambda>0$, \begin{equation}\label{petrov}\rwP{\brwD}{0}\big( |S_n| \ge \lambda \big) \le   \max (e^{-C_3 \frac{\lambda^2}{n}}, e^{-C_3 \lambda}).\end{equation}

\noindent It follows that for any  $\varepsilon >0$,  \begin{eqnarray*}\qP{\brwCh,\brwD}\Big( \max_{0\le i\le n}|\brwNd{\seq{i}}| \ge n^{\frac{1}{4}+ \varepsilon}, \, \max_{0\le i\le n}\dist(\varnothing,\seq{i})\le n^{\frac{1}{2}+ \varepsilon}\Big)
&\le&
n \, \max_{0\le k \le n^{\frac{1}{2}+ \varepsilon}} \rwP{\brwD}{0}\big( |S_k| \ge n^{\frac{1}{4}+ \varepsilon}\big)
\\
&\le&
  n \, e^{- C_3 n^\varepsilon},
\end{eqnarray*}

\noindent whose sum over $n$ converges. By using the Borel-Cantelli lemma and Lemma \ref{lem:max_dist}, we get the Corollary. 
\end{proof}

For  $\varepsilon \in (0, \frac14)$, let \begin{equation}\label{Fn} F_\varepsilon(n):= \Big\{\max_{0\le i\le n}\dist(\varnothing,\seq{i})< n^{\frac{1}{2}+\varepsilon}\Big\}. \end{equation}

\noindent By Lemma \ref{lem:max_dist}, almost surely $F_\varepsilon(n)$ holds for all large $n$. 

\begin{lemma}\label{lm:count_pair}
Let $\brwCh\not\equiv\delta_{1}$ be a probability measure on $\mathbb N$ with mean $1$ and finite second  moment, then for any $k\ge 0$,
\[
\qE{\brwCh}\bracks*{\#\setof{(i,j)}{0\le i\le j\le n,\dist(\seq{i},\seq{j})=k} \, 1_{F_n(\varepsilon)}}\le  (k+1)^2 n^{\frac12+\varepsilon} + C_4 (k+1) n^{1+2 \varepsilon},
\]
\noindent where $C_4:= \sum_{j=0}^\infty j^2 \mu(j)$.
\end{lemma}
\begin{proof} For any $u, v \in \tree^{IGW}$, we $u\preceq v$, if $u$ is an ancestor of $v$ and denote by $u \wedge v$ their most youngest common ancestor.  We consider the two cases: $\seq{i}\wedge\seq{j}\in \setof{\spine_{\ell}}{\ell\ge 0}$, $\seq{i}\wedge\seq{j}\ne\setof{\spine_{\ell}}{\ell\ge 0}$ separately.

{\bf First case:}  $\seq{i}\wedge\seq{j}= \spine_\ell$ for some $\ell\ge 0$.

Note that the subtree rooted at $\spine_{\ell}$, $\mathcal T_\ell=\setof{\gnNd}{\spine_\ell \preceq\gnNd,\spine_{\ell+1}\not\preceq\gnNd}$, is 
a critical Galton-Watson tree,
\begin{equation}\label{ex1}
\qE{\brwCh}\bracks*{
\#\setof{\gnNd\in\mathcal T_\ell}{\dist(\gnNd,\spine_\ell)=k}}
=1, \qquad \forall k\ge 1.
\end{equation}

\begin{figure}[ht]
\centering
\includegraphics[scale=0.5]{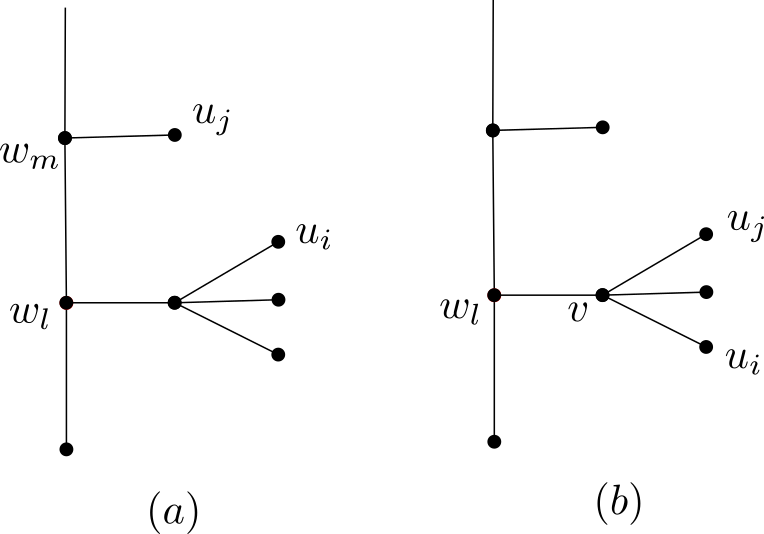}
\caption{\leftskip=1.8truecm \rightskip=1.8truecm The relative position of $\seq{i},\seq{j}$ and $\spine_ell$.}\label{figab}
\end{figure}

As is shown in \Cref{figab}(a), 
\begin{eqnarray}
&&\qE{\brwCh}\bracks*{\#\setof{(i,j)}{0\le i\le j\le n,\dist(\seq{i},\seq{j})=k,\seq{i}\wedge\seq{j}\in \setof{\spine_{\ell}}{\ell\ge 0}}\, 1_{F_n(\varepsilon)} }
\nonumber\\
&\le & 
\sum_{r=0}^k\sum_{0\le \ell<m\le \ell+k-r} 1_{\{m\le n^{\frac12+\varepsilon}\}}\,  \qE{\brwCh}\Big[  \sum_{u\in \tree_\ell, u'\in \tree_m} 1_{\{\dist(u, \spine_\ell)=r, \, \dist(u', \spine_m)=k-r- (m-\ell)\}} \Big]
\nonumber\\
&=&
\sum_{r=0}^k\sum_{0\le \ell<m\le \ell+k-r} 1_{\{m< n^{\frac12+\varepsilon}\}}
\nonumber\\
&\le&
(k+1)^2 \, n^{\frac12+\varepsilon}, \label{1ercas}
\end{eqnarray}

\noindent where the above equality follows from  \eqref{ex1} and  the independence of $\tree_\ell$ and $\tree_m$.



{\bf Second (and last) case:}   $\seq{i}\wedge\seq{j} \not\in\setof{\spine_{\ell}}{\ell\ge 0}$.

  For this case, similarly as shown in \Cref{figab}(b), let $v= u_i \wedge u_j$. On $F_n(\varepsilon)$,  $\dist(\varnothing, v) \le n^{\frac12+\varepsilon}$.  Then 
  \begin{eqnarray*}
  && \qE{\brwCh}\bracks*{\#\setof{(i,j)}{0\le i\le j\le n,\dist(\seq{i},\seq{j})=k,\seq{i}\wedge\seq{j}\not\in \setof{\spine_{\ell}}{\ell\ge 0}} \, 1_{F_n(\varepsilon)} }
  \\
  &\le& \sum_{r=0}^k\sum_{0\le \ell,t< n^{\frac{1}{2}+\varepsilon}} \qE{\brwCh} \Big[ \sum_{v\in \tree_\ell, \dist(v, \spine_\ell)=t} \, \sum_{u\wedge u'=v} 1_{\{\dist(u, v)=r, \, \dist(u', v)= k-r\}}\Big] .
  \end{eqnarray*}
  
  \noindent Conditionally on the number of children of $v$, say $j$,   by using \eqref{ex1}, the expectation of $\sum_{u\wedge u'=v} 1_{\{\dist(u, v)=r, \, \dist(u', v)= k-r\}}$ is dominated by $j^2$. Again using \eqref{ex1}, we deduce from the branching property  that $ \qE{\brwCh} \big[ \sum_{v\in \tree_\ell, \dist(v, \spine_\ell)=t} \, \sum_{u\wedge u'=v} 1_{\{\dist(u, v)=r, \, \dist(u', v)= k-r\}}\big] \le \sum_{j=0}^\infty j^2 \mu(j)$, which together with \eqref{1ercas} yield the Lemma.  
\end{proof}

\section{Proof of the upper bound in Proposition \ref{thm:mainlow}}
Before studying the capacity, we need the basic notation of Green's function:
\[
\green{\rwD}(x,y)=\green{\rwD}(y-x)
:=\rwE{\rwD}{0}\bracks*{\sum_{i=0}^\infty\mathbf 1_{{\{}\rwS{i}=y-x{\}}}}
=\sum_{i=0}^\infty\rwP{\rwD}{0}(\rwS{i}=y-x),\qquad x,y\in\mathbb Z^d.
\]
The Green function $\green{\rwD}(x)$ has the following asymptotic estimate:
\begin{lemma}[Lawler and Limic {\cite[Theorem 4.3.5]{Lawler-book-RW}}]
\label{Green_asymptotic}
Given an aperiodic and irreducible distribution $\rwD$ on $\mathbb Z^d (d\ge 3)$ with mean $0$ and covariance matrix $\rwVar{\rwD}$,
if it has finite $(d+1)$-th moment $\rwE{\rwD}{0}[|\rwS{1}|^{d+1}]<\infty$, then
\[
\green{\rwD}(x)=\frac{C_{d,\rwD}}{\rwJ{\rwD}(x)^{d-2}}+O(|x|^{1-d}),
\]
where 
$C_{d,\rwD}=\frac{\mathbf \Gamma(\frac{d}{2})}{(d-2)\pi^{d/2}\sqrt{\det\rwVar{\rwD}}},$
$\mathbf \Gamma(\cdot)$ refers to the Gamma function and $\rwJ{\rwD}(x)=\sqrt{x\cdot\rwVar{\rwD}^{-1}x}.$
\end{lemma}

Below is a lemma that connects the capacity with Green's function, which is inspired from \cite[Lemma 2.12]{bai2020capacity}.

\begin{lemma}\label{lm:GreenToCap}
Let $\rwD$ be a probability distribution in $\mathbb Z^d,\,d\ge 3$.
For any sequence $(x_n)_{n\ge 0}\in\mathbb Z^d$,  
\[
\frac{1}{n+1}\sum_{i=0}^n  1_{x_i\not\in\braces{x_{i+1},\cdots,x_n}}\rwP{\rwD}{x_i}(\tau_{\braces{x_0,\cdots,x_n}}^+=\infty)\sum_{j=0}^n \green{\rwD}(x_j,x_i)=1,
\]
where under $\rwP{\rwD}{x}$, $(S_n)$ is a random walk on $\mathbb Z^d$ started at $x$ and with jump distribution $\eta$,   and  $\tau^+_A:=\inf\setof*{i\ge 1}{S_i\in A}$ denotes as before the first returning time of $A$, for any finite $A \subset \mathbb Z^d$.
\end{lemma}
\begin{proof}
Since the random walk $(S_n)$ in dimension $d\ge 3$ is transient, for any finite set $A\subset\mathbb Z^d$ and $z\in A$, let $\sigma_A:=\sup\setof{i\ge 0}{S_i\in A}$ be the last-passage time, then
\begin{align*}
1&=\rwP{\rwD}{z}(\sigma_A<\infty)\\
&=\sum_{x\in A}\sum_{i=0}^\infty \rwP{\rwD}{z}(S_i=x)\rwP{\rwD}{x}(\tau_A^+=\infty)\\
&=\sum_{x\in A}\green{\rwD}(z,x)\rwP{\rwD}{x}(\tau_A^+=\infty).
\end{align*}
Take $A=\braces{x_0,\cdots,x_n}$ in this equation, then
\[
\sum_{i=0}^n \mathbf 1_{x_i\not\in\braces{x_{i+1},\cdots,x_n}}\rwP{\rwD}{x_i}(\tau_{\braces{x_0,\cdots,x_n}}^+=\infty) \green{\rwD}(z,x_i)=1,
\]
and the conclusion follows by summing over $z=x_0,\cdots,x_n$.
\end{proof}
Then we estimate the sum of Green's functions.
\begin{lemma}\label{lm:realG}
In dimensions $d=3,4,5$, let $\brwCh,\brwD,\rwD$ be probability distributions with the conditions in \eqref{assumption}. Then $\qP{\brwCh,\brwD}$-almost surely, \[
\min_{0\le i\le n}\sum_{j=0}^{n}
\green{\rwD}(\hdX{i},\hdX{j})\ge n^{\frac{6-d}{4}+o_{\tt as}(1)}.
\]
\end{lemma}
\begin{proof}
By \Cref{Green_asymptotic}, it suffices to show that $\qP{\brwCh,\brwD}$-almost surely,
\begin{equation}\label{eq:lowerG}
\min_{0\le i\le n}\sum_{j=0}^{n}
\frac{1}{(1+|\hdX{i}-\hdX{j}|)^{d-2}}
\ge n^{\frac{6-d}{4}+o_{\tt as}(1)}.
\end{equation}

Denote as before by $\dist(\seq{i},\seq{j})$ the graph distance between the two vertices $\seq{i},\seq{j}$ on the tree, then
\[
\hdX{i}-\hdX{j}\overset{d}{=}
S_{\dist(\seq{i},\seq{j})},
\]

\noindent where $(S_n)_{n\ge0}$ is the $\brwD$-random walk started at $0$,   independent of $\dist(\seq{i},\seq{j})$.


For any  $\varepsilon\in (0, \frac14)$, using the union bound and \eqref{petrov} we get that
\[
\qP{\brwCh,\brwD}\Big(
\cup_{0\le i, j\le n}\{|\hdX{i}-\hdX{j}|\ge n^{\varepsilon}\sqrt{1+\dist(\seq{i},\seq{j})}\}\Big)
\le (n+1)^2 \, e^{- C_3 \, n^{\varepsilon}}.
\]
By the Borel-Cantelli lemma, almost surely  the above  event cannot happen infinitely often. Thus to prove \eqref{eq:lowerG}, it suffices to show that $\qP{\brwCh,\brwD}$-almost surely,
\begin{equation*}
\min_{0\le i\le n}\sum_{j=0}^{n}
\frac{1}{(1+\dist(\seq{i},\seq{j}))^{\frac{d-2}{2}}}
\ge n^{\frac{6-d}{4}+o_{\tt as}(1)},
\end{equation*}
or more generally,   for any $\alpha>0$, $\qP{\brwCh,\brwD}$-almost surely 
\begin{equation}\label{eq:min_sum_dist}
\min_{0\le i\le n}\sum_{j=0}^{n}
\frac{1}{(1+\dist(\seq{i},\seq{j}))^{\alpha}}
\ge n^{1-\frac{\alpha}{2}+o_{\tt as}(1)}.
\end{equation}

Observe that $\min_{0\le i\le n}\sum_{j=0}^{n}
\frac{1}{(1+\dist(\seq{i},\seq{j}))^{\alpha}}
 \ge (n+1)\pars*{1+ \max_{0\le i,j\le n}\dist(\seq{i},\seq{j})}^{-\alpha} \ge  (n+1)\pars*{1+2\max_{0\le i\le n}\dist(\varnothing,\seq{i})}^{-\alpha}$,  then  \eqref{eq:min_sum_dist} follows from \Cref{lem:max_dist}. 
\end{proof}

\medskip
{\noindent\it Proof of the upper bound in Proposition \ref{thm:mainlow}:}  Applying \Cref{lm:GreenToCap} to $\braces{\brwNd{\seq{0}},\cdots,\brwNd{\seq{n}}}$, we deduce from the definition of the $\eta$-capacity that
\begin{align*}
\capa \gnR[0,n]&=\sum_{i=0}^n   1_{\brwNd{\seq{i}}\not\in\braces{\brwNd{\seq{i+1}},\cdots,\brwNd{\seq{n}}}}  \, \rwP{\rwD}{\brwNd{\seq{i}}}\left(\tau_{\braces{\brwNd{\seq{0}},\cdots,\brwNd{\seq{n}}}}^+=\infty \,|\, \braces{\brwNd{\seq{0}},\cdots,\brwNd{\seq{n}}}\right)\\
&\le\frac{n+1}{\min_{0\le i\le n}\sum_{j=0}^n \green{\rwD}(\brwNd{\seq{i}},\brwNd{\seq{j}})},
\end{align*}
and the conclusion follows from \Cref{lm:realG}.  $\hfill\Box$

\section{Proof of the lower bound in Proposition \ref{thm:mainlow}}

For the lower bound, our main tool is the following lemma.
\begin{lemma}[{\cite[Lemma 2.11]{bai2020capacity}}]
\label{capA}
Let $d\ge 3$ {and} $\rwD$ be any probability distribution on $\mathbb Z^d$. For any finite set $A\subset \mathbb Z^d$ and $k\in\mathbb N_+$, 
\[\capa A\ge\frac{\#A}{k+1}-\frac{\sum_{x,y\in A}\green{\rwD}(x,y)}{k(k+1)}.\]
\end{lemma}

According to this lemma, the capacity $\capa \gnR[0,n]$ can be bounded below by estimates of $\#\gnR[0,n]$ and the sum of Green's functions.
We start with $\#\gnR[0,n]$. Let
\[
L^x_n:=\sum_{i=0}^n\mathbf 1_{\braces*{\brwNd{\seq{i}}=x}},\qquad \forall x\in\mathbb Z^d,\,n\ge 0,
\]
denote the local times, then we can write the range as
\[R[0,n]=\setof{x\in\mathbb Z^d}{{L^x_n\ge 1}}.\]
The following second moment estimate for local times is inspired by the proof of Le Gall and Lin \cite[Lemma 3]{LeGall-Lin-lowdim}.

\begin{lemma}\label{pp:caplocal}
Let $d\ge 3$. With the conditions in \eqref{assumption}, $\qP{\brwCh,\brwD}$-almost surely, as $n \to \infty$, 
\[
\sum_{x\in\mathbb Z^d} (L^x_n)^2\le  n^{\max(\frac{8-d}{4}, 1)+o_{\tt as}(1)}. 
\]
\end{lemma}
\begin{proof} Let $\varepsilon \in (0, \frac14)$ and recall the event $F_\varepsilon(n)$ defined in \eqref{Fn}. We are going to prove that for all $n\ge 1$, \begin{equation} \label{2ndmoment}\sum_{x\in\mathbb Z^d, |x|\le n} \qE{\brwCh,\brwD} [(L^x_n)^2\, 1_{F_\varepsilon(n)}]\le C_5 \,  n^{\max(\frac{8-d}{4}, 1)+ 4 \varepsilon}. 
\end{equation}

Admitting for the moment \eqref{2ndmoment} we can give the proof of Lemma \ref{pp:caplocal}. Let $\xi_n:=  \sum_{x\in\mathbb Z^d, |x|\le n}  (L^x_n)^2$,  $\gamma:= \max(\frac{8-d}{4}, 1) + 5\varepsilon$ and $n_j:= 2^j$ for $j\ge 1$. By Markov's inequality, \eqref{2ndmoment} implies that for all $j\ge 1$, $$  \qP{\brwCh,\brwD} 
\big( \xi_{n_j} \ge n_{j-1}^\gamma, F_\varepsilon(n_j)\big) \le C_5 \frac{n_j^{\gamma-\varepsilon}}{n_{j-1}^\gamma} \le C_6 \, 2^{-\varepsilon j}.$$

\noindent The Borel-Cantelli lemma says that almost surely for all large $j$, either $\xi_{n_j} < n_{j-1}^\gamma$ or $F_\varepsilon(n_j)$ does not hold. However by  Lemma \ref{lem:max_dist}, almost surely $F_\varepsilon(n_j)$ holds for all large $j$, hence we have proved that almost surely for all large $j$,   $\xi_{n_j} < n_{j-1}^\gamma$.  On the other hand, by Corollary \ref{cor:max_v}, 
almost surely for all large $j$, $L_{n_j}^x = 0$ for all $|x|> n_j$, hence $\sum_{x\in\mathbb Z^d}  (L^x_{n_j})^2 = \xi_{n_j} < n_{j-1}^\gamma$. Then by monotonicity for all large $n$, $\sum_{x\in\mathbb Z^d}  (L^x_{n})^2   < n^\gamma$ a.s. Since $\varepsilon$ can be arbitrarily small, we have proved Lemma \ref{pp:caplocal}.

It remains to show \eqref{2ndmoment}. To this end,   we denote the transition probabilities for a $\brwD$-walk  $(S_n)_{n\ge0}$ by
\[
\pi_m(x):=\rwP{\brwD}{0}(S_m=x), \qquad m\ge0, x\in \mathbb Z^d, 
\]
for simplicity.
Then there exists a constant $C_7>0$ depending on $d$ and $\theta$ such that for  all $ x\in\mathbb Z^d$ and  $m\ge 0$, 
\begin{align}
\label{eq:pi1}
&\pi_m(x)\le C_7\, (1+|x|)^{-d},\\
\label{eq:pi2}
&\pi_m(x)\le C_7\, (1+ m)^{-\frac{d}{2}},\\
\label{eq:pi3}
&\sum_{x\in\mathbb Z^d}\pi_m(x)=1.\end{align}
where \eqref{eq:pi1} follows from \cite[Proposition 2.4.6]{Lawler-book-RW}, and \eqref{eq:pi2} follows from \cite[p.24]{Lawler-book-RW}. (In \cite{Lawler-book-RW}, $\brwD$ is also required to be aperiodic, but since we only need an upper bound, these results can be easily extended to periodic cases.)
Then we decompose the second moment in \eqref{2ndmoment} as
\[
\sum_{x\in\mathbb Z^d, |x|\le n}\qE{\brwCh,\brwD}\bracks*{(L^x_n)^2 \, 1_{F_\varepsilon(n)}}=\sum_{x\in\mathbb Z^d, |x|\le n} \sum_{i,j=0}^n\hdP\Big(\brwNd{\seq{i}}=\brwNd{\seq{j}}=x, F_\varepsilon(n)\Big).
\]

For notational brevity,  we write $\gnNd=\gnNd_i\wedge \gnNd_j$ for the youngest common ancestor of $\gnNd_i,\gnNd_j$, and $y=\brwNd{\gnNd}$ for the spatial location of $\gnNd$. We also write $a=\dist(\varnothing,\gnNd)$, $b=\dist(\gnNd,\gnNd_i)$ and $c=\dist(\gnNd,\gnNd_j)$ for the graph distances between these particles, as shown in \Cref{fig3}.
\begin{figure}[ht]
\centering
\includegraphics[scale=0.5]{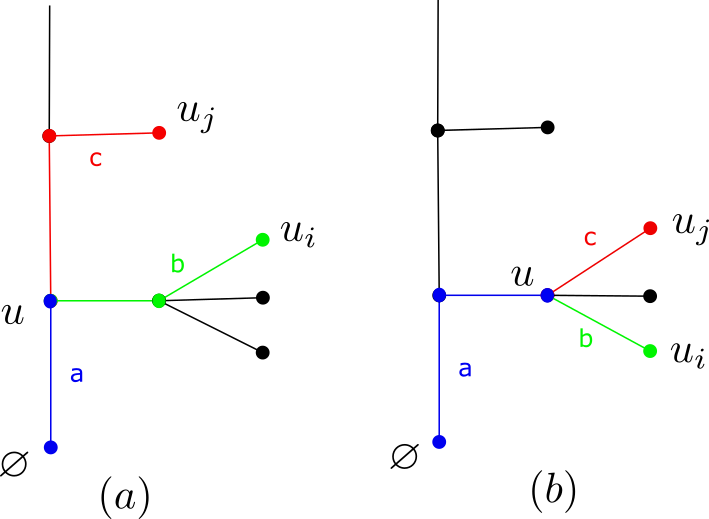}
\caption{\leftskip=1.8truecm \rightskip=1.8truecmAn illustration for the relative positions of $u,u_i,u_j$.}\label{fig3}
\end{figure}

We assume without loss of generality that $b\ge c$, then 
\begin{equation}\label{eq:dist}
b\ge \frac{1}{2}\dist(\gnNd_i,\gnNd_j).
\end{equation}

\noindent Therefore (keeping in mind that $a, b, c$ depend on $u_i, u_j$), 
\begin{align*}
&\qE{\brwCh,\brwD}\bracks*{(L^x_n)^2 \, 1_{F_\varepsilon(n)}}\\
=&\sum_{i,j=0}^n\hdP(\brwNd{\seq{i}}=\brwNd{\seq{j}}=x, F_\varepsilon(n))\\
=&\sum_{i,j=0}^n\sum_{y\in\mathbb Z^d}\hdP(\brwNd{\gnNd}=y,\brwNd{\seq{i}}=\brwNd{\seq{j}}=x, F_\varepsilon(n))\\
=&\hdE\bracks*{\sum_{i,j=0}^n\sum_{y\in\mathbb Z^d}\pi_a(y)\pi_b(x-y)\pi_c(x-y) \, 1_{F_\varepsilon(n)}}\\
=&A+B, 
\end{align*}

\noindent where \begin{eqnarray*}
A &:=&\hdE\bracks*{\sum_{i,j=0}^n\sum_{|y|\ge |x|/2}\pi_a(y)\pi_b(x-y)\pi_c(x-y)\, 1_{F_\varepsilon(n)}},
\\
B&:=&\hdE\bracks*{\sum_{i,j=0}^n\sum_{|y|< |x|/2}\pi_a(y)\pi_b(x-y)\pi_c(x-y)\, 1_{F_\varepsilon(n)}} .\end{eqnarray*}

For $A$, we use \eqref{eq:pi1} for $\pi_a$, \eqref{eq:pi2} and \eqref{eq:dist} for $\pi_b$ and \eqref{eq:pi3} for $\pi_c$, then
\begin{align*}
A&\le C_7^2\, \hdE\bracks*{\sum_{i,j=0}^n\pars*{1+\frac{|x|}{2}}^{-d}\pars*{1+\frac{1}{2}\dist({\seq{i}},{\seq{j}})}^{-\frac{d}{2}}\sum_{y\in\mathbb Z^d}\pi_c(x-y) \, 1_{F_\varepsilon(n)}}\\
&=C_7^2\, \hdE\bracks*{\sum_{i,j=0}^n\pars*{1+\frac{|x|}{2}}^{-d}\pars*{1+\frac{1}{2}\dist({\seq{i}},{\seq{j}})}^{-\frac{d}{2}} \, 1_{F_\varepsilon(n)}} \\
&= C_7^2\, \pars*{1+\frac{|x|}{2}}^{-d}
\sum_{k= 0}^\infty \pars*{1+\frac{k}{2}}^{-\frac{d}{2}}
\hdE\bracks*{\#\setof{0\le i,j\le n}{\dist({\seq{i}},{\seq{j}})=k} \, 1_{F_\varepsilon(n)}}.
\end{align*}

\noindent 
Note that 
\begin{eqnarray*}
&&\sum_{k=0}^\infty\pars*{1+\frac{k}{2}}^{-\frac{d}{2}}
\hdE\bracks*{\#\setof{0\le i,j\le n}{\dist({\seq{i}},{\seq{j}})=k} \, 1_{F_\varepsilon(n)}}
\\
&\le&
\sum_{0\le k \le  n^{\frac12}}\pars*{1+\frac{k}{2}}^{-\frac{d}{2}}
\hdE\bracks*{\#\setof{0\le i,j\le n}{\dist({\seq{i}},{\seq{j}})=k} \, 1_{F_\varepsilon(n)}} +
\pars*{1+\frac12 n^{\frac12}}^{-\frac{d}{2}} n^2,
\end{eqnarray*}
 
\noindent which by  \Cref{lm:count_pair},  is  further bounded by \begin{eqnarray*} \sum_{0\le k \le  n^{\frac12}} \pars*{1+\frac{k}{2}}^{-\frac{d}{2}} \big[(k+1)^2 n^{\frac12+\varepsilon}+ C_4 (k+1) n^{1+2\varepsilon}\big] + \pars*{1+ \frac12 n^{\frac12}}^{-\frac{d}{2}} n^2
\le C_8\, n^{\max(\frac{8-d}{4}, 1)+ 3\varepsilon }.
\end{eqnarray*}

  Therefore we have proved that for any $x\in \mathbb Z^d$ and $n\ge1$, $$A \le C_7^2\, C_8\, \pars*{1+\frac{|x|}{2}}^{-d} n^{\max(\frac{8-d}{4}, 1)+ 3\varepsilon } .$$

We may deal with the term $B$ in a similar way. If $|y|<\frac{|x|}{2}$, then $|x-y|\ge\frac{|x|}{2}$, so we use \eqref{eq:pi3} for $\pi_a$, \eqref{eq:pi2} and \eqref{eq:dist} for $\pi_b$ and \eqref{eq:pi1} for $\pi_c$,
\begin{align*}
B&\le C_7^2\,  \hdE\bracks*{\sum_{i,j=0}^n\sum_{y\in\mathbb Z^d}\pi_a(y)\pars*{1+\frac{1}{2}\dist({\seq{i}},{\seq{j}})}^{-\frac{d}{2}}\pars*{1+\frac{|x|}{2}}^{-d} \, 1_{F_\varepsilon(n)}}\\
&=C_7^2\, \hdE\bracks*{\sum_{i,j=0}^n\pars*{1+\frac{|x|}{2}}^{-d}\pars*{1+\frac{1}{2}\dist({\seq{i}},{\seq{j}})}^{-\frac{d}{2}} \, 1_{F_\varepsilon(n)}}\\
&\le C_7^2\, C_8\, \pars*{1+\frac{|x|}{2}}^{-d} n^{\max(\frac{8-d}{4}, 1)+ 3\varepsilon }.
\end{align*}

\noindent Then for any $x\in \mathbb Z^d$, we have
\[
\hdE\bracks*{(L^x_n)^2 \, 1_{F_\varepsilon(n)}}=A+B\le 2 C_7^2\, C_8\, \pars*{1+\frac{|x|}{2}}^{-d} n^{\max(\frac{8-d}{4}, 1)+ 3\varepsilon }.
\]
Taking the sum over $|x|\le n$ gives \eqref{2ndmoment}. This completes the proof of  Lemma \ref{pp:caplocal}. \end{proof}

From this lemma we deduce an almost-sure lower bound for $\#\gnR[0,n]$:
\begin{proposition}\label{pp:caprange}
For $d\ge 3$, let $\brwCh,\brwD$ be probability distributions with the conditions in \eqref{assumption}, then $\qP{\brwCh,\brwD}$-almost surely  for all large $n$, 
\[\#\gnR[0,n]\ge n^{\min(\frac{d}{4}, 1)+o_{\tt as}(1)}.\]
\end{proposition}
\begin{proof}
By definition,
\[
\sum_{x\in \gnR[0,n]}L^x_n=n+1,
\]
then by Cauchy-Schwarz' inequality,
\[
\#\gnR[0,n] \ge \frac{(n+1)^2}{\sum_{x\in\mathbb Z^d}(L^x_n)^2}.
\]
 We conclude by  Lemma \ref{pp:caplocal}.\end{proof}

\begin{lemma}\label{lm:green_upper}
For $d=3,4,5$, let $\brwCh,\brwD,\rwD$ be probability distributions with the conditions in \eqref{assumption}, then $\qP{\brwCh,\brwD}$-almost surely for all large $n$,
\[
\sum_{x,y\in\gnR[0,n]}\green{\rwD}(x,y)\le 
\left\{\begin{array}{ll}
 n^{\frac{5}{4}+o_{\tt as}(1)}, &d=3\\ \\
 n^{\frac{10-d}{4}+o_{\tt as}(1)}, &d=4,5
\end{array}\right..
\]
\end{lemma}
\begin{remark}
One would expect that the sum of Green's functions is monotone decreasing in $d$ with a unified asymptotic formula. However, in dimension $d=3$, $R[0,n]$ contains considerably less points than that in $d\in\{4,5\}$. Therefore, we have different results and proofs for the case $d=3$ and the case $d\in\{4,5\}$. $\hfill\Box$
\end{remark}
\begin{proof} Let $\varepsilon \in (0, \frac1{12})$ be small.

For $d=3$, by \Cref{cor:max_v}, $\qP{\brwCh,\brwD}$-almost surely for all large $n$,  
$$\sum_{x,y\in\gnR[0,n]}\green{\rwD}(x,y)
 \le \sum_{|x|,|y|\le n^{\frac14+\varepsilon}}\green{\rwD}(x,y) 
 \le C_9\, n^{\frac34+ 3\varepsilon},
$$

\noindent where the last inequality follows from the asymptotic behaviors of $G_\eta$ given in   Lemma \ref{Green_asymptotic}. This proved the case $d=3$.

For $d\in \{4,5\}$, recall the event $F_\varepsilon(n)$ defined in \eqref{Fn}.  Since $\sum_{x,y\in\gnR[0,n]}\green{\rwD}(x,y) \le  \sum_{i,j=0}^n \green{\rwD}(\brwNd{\seq{i}},\brwNd{\seq{j}}) $, we have $$\qE{\brwCh,\brwD}\bracks*{\sum_{x,y\in\gnR[0,n]}\green{\rwD}(x,y) \, 1_{F_\varepsilon(n)}} 
\le
\sum_{i,j=0}^n \qE{\brwCh,\brwD}\bracks*{\green{\rwD}(\brwNd{\seq{i}},\brwNd{\seq{j}}) \, 1_{F_\varepsilon(n)}}.$$

\noindent Using \Cref{Green_asymptotic} and the fact that $ \brwNd{\seq{i}}-\brwNd{\seq{j}} $ is distributed as $S_{\dist(u_i, u_j)}$ with $S$ a $\brwD$-random walk independent of $ \dist(u_i, u_j)$, we deduce from the local limit theorem for $S$ that   \begin{align*}
\qE{\brwCh,\brwD}\bracks*{\sum_{x,y\in\gnR[0,n]}\green{\rwD}(x,y) \, 1_{F_\varepsilon(n)}}
\le& C_{10}\, \sum_{i,j=0}^n\, \qE{\brwCh,\brwD}\bracks*{\frac{1}{(1+|\brwNd{\seq{i}}-\brwNd{\seq{j}}|)^{d-2}} \, 1_{F_\varepsilon(n)}}\\
\le & C_{11} \, \sum_{i,j=0}^n \, \qE{\brwCh,\brwD}\bracks*{\frac{1}{(1+\dist({\seq{i}},{\seq{j}}))^{\frac{d-2}{2}}}\, 1_{F_\varepsilon(n)} }\\
=& C_{11} \, {\sum_{k=0}^{\infty}\frac{\qE{\brwCh,\brwD}\bracks*{\#\setof{0\le i,j\le n}{\dist(\seq{i},\seq{j})=k} \, 1_{F_\varepsilon(n)}} }{(1+k)^{\frac{d-2}{2}}}}.
\end{align*}

The above sum over $k$ is less than $${\sum_{0\le k \le \sqrt n}\frac{\qE{\brwCh,\brwD}\bracks*{\#\setof{0\le i,j\le n}{\dist(\seq{i},\seq{j})=k} \, 1_{F_\varepsilon(n)}} }{(1+k)^{\frac{d-2}{2}}}} + \frac{n^2}{(1+ \sqrt n)^{\frac{d-2}{2}}},
$$

\noindent which by \Cref{lm:count_pair} is further bounded by 
$$ \sum_{0\le k \le \sqrt n} \big[(k+1)^{3-\frac{d}2} n^{\frac12+\varepsilon} + C_4 (k+1)^{2-\frac{d}{2}} n^{1+2\varepsilon}\big]+ n^{\frac{10-d}4}
\le C_{12} n^{\frac{10-d}4 +2\varepsilon}.
$$

Then we have shown that for all $n\ge 1$, $$ \qE{\brwCh,\brwD}\Big(\sum_{x,y\in\gnR[0,n]}\green{\rwD}(x,y) \, 1_{F_\varepsilon(n)}\Big)
\le C_{11}\, C_{12} \, n^{\frac{10-d}4 +2\varepsilon}.$$

\noindent Similarly to the proof of Lemma \ref{pp:caplocal}, we use the Borel-Cantelli lemma and the fact that  $F_\varepsilon(n)$ holds eventually for all large $n$ (Lemma \ref{lem:max_dist}), to get that a.s. for all large $n$, $\sum_{x,y\in\gnR[0,n]}\green{\rwD}(x,y) \le n^{\frac{10-d}4 +3\varepsilon}$. Since $\varepsilon$ can be arbitrarily small, we get the Lemma for the case $d\in \{4, 5\}$.  \end{proof}

{\noindent\it Proof of the lower bound in \Cref{thm:mainlow}:}  Let  $d\in \{3,4,5\}$. Let $\varepsilon \in (0, \frac1{12})$ be small. By  \Cref{pp:caprange} and \Cref{lm:green_upper}, we see that $\qP{\brwCh,\brwD}$-almost surely for all large $n$, $\#\gnR[0,n]\ge n^{\min(\frac{d}{4}, 1) -\varepsilon}$, and $\sum_{x,y\in\gnR[0,n]}\green{\rwD}(x,y)\le 
\left\{\begin{array}{ll}
 n^{\frac{5}{4}+ \varepsilon}, &d=3\\   n^{\frac{10-d}{4}+\varepsilon}, &d=4,5
\end{array}\right..$

Applying  \Cref{capA} to $A=\gnR[0,n]$ with $k= \lfloor 2 n^{\frac12 + 2 \varepsilon}\rfloor$ if $d=3$ and $k=\lfloor 2 n^{\frac{6-d}4+ 2 \varepsilon}\rfloor$ if $d\in \{4, 5\}$, we get that  $\qP{\brwCh,\brwD}$-almost surely for all large $n$,  $\capa \gnR[0,n]\ge \frac15\, n^{\frac{d-2}{4}-3 \varepsilon}$. Since $\varepsilon$ can be arbitrarily small,  this gives the  lower bound in \Cref{thm:mainlow}. $\hfill\Box$

\end{document}